\documentclass[10pt]{amsart}
\usepackage{amsmath,graphicx,amssymb,amsthm}
\usepackage{color}
\def\A{\mathcal{A}}
\def\B{\mathcal B}

\def\B{\mathcal B}

\def\amslatex{$\mathcal{A}\kern-.1667em\lower.5ex\hbox{$\mathcal{M}$}\kern-.125em\mathcal{S}$-\LaTeX}

\newtheorem{set}{set}[section]
\newtheorem{Corollary}[set]{Corollary}
\newtheorem{Definition}[set]{Definition}

\newtheorem{Lemma}[set]{Lemma}

\newtheorem{Proposition}[set]{Proposition}
\newtheorem{Remark}[set]{Remark}
\newtheorem{Theorem}[set]{Theorem}
\newcommand{\define}{\mathrel{\hbox{$\equiv$\hskip -.90em \lower .47ex \hbox{$\leftharpoondown$}}}}
\newcommand{\enifed}{\mathrel{\hbox{$\equiv$\hskip -.90em \lower .47ex \hbox{$\rightharpoondown$}}}}

\begin{document}
\title[Poisson Processes in Free Probability] {Poisson Processes in Free Probability}
\author{Guimei An}
\address{School of Mathematical Sciences and LPMC \\
Nankai University, Tianjin 300071, China} \email[Guimei
An]{angm@nankai.edu.cn}

\author{Mingchu Gao}
\address{Department of Mathematics\\
Louisiana College, Pineville, LA 71359, USA} \email[Mingchu
Gao]{mingchu.gao@lacollege.edu}

\maketitle
\begin{abstract}

 We prove a multidimensional Poisson limit theorem in free probability, and define joint free Poisson distributions in a non-commutative probability space.  We define (compound) free Poisson process explicitly,
 similar to the definitions of (compound) Poisson processes in classical probability. We proved that the sum of finitely many freely independent compound free Poisson processes is a compound free Poisson processes. We give a step by step procedure for constructing a (compound) free Poisson process. A Karhunen-Loeve expansion theorem for centered free Poisson processes is proved. We generalize free Poisson processes to a notion of free Poisson random measures (which is slightly different from the previously defined ones in free probability, but more like an analogue of classical Poisson random measures). Then we develop the integration theory of real-valued functions with respect to a free Poisson random measure, generalizing the classical integration theory to the free probability case. We  find that the integral of a function (in certain spaces of functions) with respect to a free Poisson random measure has a compound free Poisson distribution. For centered free Poisson random measures, we can get a simpler and more beautiful integration theory.
\end{abstract}

{\bf Key Words.} Free Probability,  Free Poisson Processes, Integration with respect to free Poisson random measures.

{\bf 2010 MSC} 46L54
\section*{Introduction}

The theory of stochastic processes is a very important branch in
classical probability with wide applications in engineering and
finance (\cite{DJ} and \cite{TK}).  In free probability theory,
stochastic processes have been studied since 1990's. The most
popular and important stochastic process in classical probability is
Brownian motion (the Wiener process). The counterpart of Brownian
motion in free probability is the free Brownian motion. The free
Brownian motion and stochastic analysis with respect to the free
Brownian motion have been studied thoroughly (\cite{PBi},
\cite{BS1}, \cite{BS2} etc.). Anshelevich \cite{MA1} developed an
integration theory of bi-processes with respect to (additive)
non-commutative stochastic measures. Free infinite divisibility and
free Levy processes and stochastic integration with respect to a
free Levy process were studies in \cite{BnT}. Certain stochastic
differential equations driven by free Levy processes were studied in
\cite{MG1} and \cite{MG2}.

It is well known that Poisson distributions  form  a class of the
most prominent processes in classical probability beyond normal
distributions (Lecture 12 in \cite{NS}), and free Poisson processes
form a class of the most important processes with free increment in
free probability after free Brownian motion (\cite{MA}). But free
Poisson distributions and processes have not been investigated
thoroughly. In this paper, we study some interesting questions on
free Poisson distributions and free Poisson processes.

{\bf A free Poisson Limit Theorem}. The counterpart of normal distributions in free probability is semicircle distributions. There is a semicircle limit theorem called {\sl free central limit theorem}(Theorem 8.10 in \cite{NS}).

Very similarly, a free Poisson distribution can be realized as the limit in distribution of a sequence of simple distributions (Proposition 12.11, Definition 12.12 in \cite{NS}). Nica and Speicher  presented a multidimensional central limit theorem (Theorem 8.17 in \cite{NS}). Roughly speaking, the theorem states that a joint semicircle distribution can be realized as the limit in distribution of a sequence of families of random variables. In this paper, we proved a multidimensional free Poisson limit theorem (Theorem 2.4). Therefore, a  joint free Poisson distribution can be defined as the limit in distribution of certain sequence of families of elements (Definition 2.6).

{\bf Free Poisson processes}. A construction of free Poisson process
with all free cumulants equal to $1$  was given in Section 4.2 in
\cite{MA}, but no definition of free Poisson processes was give
there.  Anshelevich gave a description of free Poisson processes as
``{\sl A process with stationary freely independent increments such
that the increments have free Poisson distributions is the free
Poisson process}" (4.2 in \cite{MA2}). In this paper, we give a
 definition of free Poisson process (Definition 3.1), an
analogue of a classical Poisson process. We provide a step-by-step
procedure for constructing a free Poisson process (Theorem 3.2). Nica and
Speicher gave the definition of  compound free Poisson distributions
in 12.16 of \cite{NS}. We generalize  free Poisson processes to the
compound case (Definition 3.5), and give a similar procedure for constructing
a compound free Poisson process (Theorem 3.6). In classical
probability, the sum of two independent Poisson processes is a
Poisson process (Section 2.3 in \cite{RG}). We prove in this paper
that the sum of finitely many  freely independent compound free
Poisson processes is a compound free Poisson process (Theorem 3.7),
and conditions under  which the sum of two freely independent free
Poisson processes is a free Poisson process are given (Corollary
3.8).

The Karhunen-Loeve expansion of a stochastic process is a significant result in classical stochastic processes (\cite{DJ}). Roughly speaking, the expansion says that under certain conditions, a stochastic process can be represented as  an infinite series of the products of random variables and deterministic functions $$X_t=\sum_{i=1}^\infty X_i\phi_i(t),0<t\le T,$$
where $X_i, i=1,2, \cdots,$ are uncorrelated random variables ($E(X_iX_j)=\delta_{i,j}\lambda_i$), and
 $\{\phi_i:i=1, 2, \cdots\}$ is an orthonormal basis of $L^2([0,T])$, $T>0$
(Theorem 5.3 or \cite{AA}).  In this paper, we present a Karhunen-Loeve expansion for a centered $L^2$-continuous free Poisson process in a $W^*$-probability space $(\A,\varphi)$ with precise formulas for $\phi_i(t)$ and $\lambda_i$ (Theorem 4.5).

{\bf Integration with respect to a free Poisson random measure}. Stochastic integration with respect to a non-commutative stochastic measure was studied by several mathematicians. Anshelevich \cite{MA1} defined a non-commutative stochastic measure as follows.
\begin{Definition}[Definition 1 in \cite{MA1}] A non-commutative stochastic measure is a map from the set of all finite half-open intervals $I=[a,b)\subset [0,\infty)$ to the self-adjoint part of a $W^*$-probability space $(\A,\varphi)$, $I\mapsto X(I)$,  with three properties.
\begin{enumerate}
\item Additivity. $I_1\cap I_2=\emptyset, I_1\cup I_2=J\Rightarrow X(I_1)+X(I_2)=X(J)$.
\item Stationary. The distribution of $X(I)$ dependents only on $|I|$.
\item Free increments. If $I_1, I_2, \cdots, I_n$ are mutually disjoint intervals, then $$X(I_1), X(I_2), \cdots, X(I_n)$$ are freely independent.
\end{enumerate}
\end{Definition}
Then Anshelevich \cite{MA1} defined the integral of a bi-process $U$ in $\A\otimes \A^{op}$ with respect to a non-commutative stochastic measure (\cite{MA1}). Glockner, Schurmann, and Speicher \cite{GSS} gave a definition in a $*$-probability space similar to the above Definition 0.1, and  named it  a {\sl free white noise}.

Barndorff-Nielson and Thorjornsen \cite{BnT} defined free Poisson random measures in a more general setting.
\begin{Definition}[Definition 6.7 in \cite{BnT}] Let $(\Theta, \mathcal{E}, \nu)$ be a measure space, and
$\mathcal{E}_0=\{E\in \mathcal{E}:\nu(E)<\infty\}$. A free Poisson random measure is a map $M$ from $\mathcal{E}_0$ into the cone of all non-negative operators of a $W^*$-probability space $(\A,\varphi)$ with the following properties.
\begin{enumerate}
\item $\forall E\in \mathcal{E}_0, M(E)$ has a free Poisson distribution $\kappa_n(M(E))=\nu(E), n=1, 2, \cdots$.
\item If $E_1, E_2, \cdots, E_n$ are mutually disjoint sets in $\mathcal{E}_0$, then $M(E_1), M(E_2), \cdots, M(E_n)$ are freely independent.
\item If $E_1, E_2, \cdots, E_n$ are mutually disjoint sets in $\mathcal{E}_0$, then $M(\cup_{i=1}^nE_i)=\sum_{i=1}^nM(E_i)$.
\end{enumerate}
\end{Definition}
The authors of \cite{BnT} also gave an existence theorem for free
Poisson random measures (Theorem 6.9 in \cite{BnT}), and defined the
integral of a $L^1(\Theta, \nu)$ function with respect to a free
Poisson random measure (Definition 6.19 in \cite{BnT}).

A definition of free Poisson random measures, very similar to Definition 0.2 above,  was given in \cite{BP}.  The authors of \cite{BP} studied multiple integrals of a special kind of functions with respect to a free Poisson random measure, and proved a semicircle limit theorem for free Poisson multiple integrals (Theorem 4.1 in \cite{BP}).

In this paper, we define free Poisson random measures via a sightly
different way from the others mentioned above in a $W^*$-probability
space. We do not require that operators $X_E$, for
$E\subset \mathbb{R}$ of finite measure, be non-negative, but
self-adjoint only (Definition 5.1).  Our definition of free Poisson
random measures is more like an analogue to that in classical
probability theory (Section 9.3 in \cite{TK}).  We define the
integral $X(f)$ of a function $f\in L^1(\mathbb{R})\cap
L^2(\mathbb{R})$ with respect to a free Poisson random measure
(Theorem 5.4). We prove a limit and free stochastic integration
exchange formula $$\lim_{n\rightarrow
\infty}\int_\mathbb{R}f_n(t)dX_E(t)=\int_\mathbb{R}\lim_{n\rightarrow
\infty}f_n(t)dX_E(t)$$ (Theorem 5.5). If $X_E\ge 0$, for every
$E\subset \mathbb{R}$ of finite measure, then the integration
operator $X:L^1_\mathbb{R}(\mathbb{R})\rightarrow L^1(\A,\varphi)$
is contractive (Theorem 5.7), where $L^1_\mathbb{R}(\mathbb{R})$ is
the space of all real-valued $L^1$-functions on $\mathbb{R}$. When
we focus on $L^{\infty -} =\cap_{n\ge 1}L^n(\mathbb{R})$, we find
that the integral $X(f)$ of $f\in L^{\infty -}$ has a compound free
Poisson distribution (Theorem 5.9). For a centered free Poisson
random measure (Definition 6.1), the integration operator $X$ is an
isometry from $L^2(\mathbb{R})$ into $L^2(\A,\varphi)$ (Lemma 6.3).
The integration operator $X$ with respect to a centered free Poisson
random measure can be extended to a bounded operator from
$L^1(\mathbb{R})$ into $L^1(\A,\varphi)$ with  norm less than or
equal to $2$ (Lemma 6.4).
\section{Preliminaries}

In this section we recall some basic concepts and results in free probability used in sequel or mentioned previously. The reader is referred to \cite{NS} and \cite{VDN} for the basics on free probability,  and to \cite{KR} for operator algebras.

{\bf Non-commutative Probability spaces}. A non-commutative
probability space is a pair $(\A,\varphi)$ consisting of a unital
algebra $\A$ and a unital linear functional $\varphi$ on $\A$. When
$\A$ is a $*$-unital algebra, $\varphi$ should be positive, i. e.
$\varphi(a^*a)\ge 0,\forall a\in \A$. A $C^*$-probability space
$(\A,\varphi)$ consists of a unital $C^*$-algebra and a state
$\varphi $ on $\A$. A $W^*$-probability space $(\A,\varphi)$
consists of a finite von Neumann algebra $\A$ and a faithful normal
tracial state $\varphi$ on $\A$. An element $a\in \A$ is called a
(non-commutative) random variable. $\varphi (a^n)$ is called the
$n$-th moment of $a$, $n=1, 2, \cdots$. Let $\mathbb{C}[X]$ be the
complex algebra of all polynomials of an indeterminate $X$. The
linear function $\mu_a:\mathbb{C}[X]\rightarrow \mathbb{C}$,
$\mu_a(P(X))=\varphi (P(a)), \forall P\in \mathbb{C}[X]$, is called
the distribution (or law) of $a$. A sequence $\{a_n\}$ of random
variables $a_n\in  (\A_n,\varphi_n)$ converges in distribution to
$a\in (\A,\varphi)$ if $$\lim_{n\rightarrow
\infty}\varphi_n(a_n^m)=\varphi (a^m), \forall m\ge 1.$$

{\bf Joint Distributions.}  Let $\mathbb{C}\langle X_1, X_2, \cdots, X_s\rangle$ be the unital algebra freely generated by $s$ non-commutative indeterminates  $X_1, X_2, \cdots, X_s$, and $a_1, a_2, \cdots, a_s\in \A$, where $(\A,\varphi)$ is a non-commutative probability space. The family $\{\varphi (a_{i_1}a_{i_2}\cdots a_{in}): 1\le i_1\le i_2\le \cdots \le i_n\le s, n\ge 1\}$ is called the  family of joint moments of $a_1, a_2, \cdots, a_s$. The linear functional $\mu:\mathbb{C}\langle X_1, X_2, \cdots, X_s\rangle \rightarrow \mathbb{C}$ defined by $$\mu(P)=\varphi (P(a_1, a_2, \cdots, a_s)), \forall P\in \mathbb{C}\langle X_1, \cdots, X_s \rangle,$$ is called the joint distribution of $a_1, a_2, \cdots, a_s$. Similar to the single variable case, we can define the limit in distribution of a sequence of families of random variables.

{\bf Free independence.}   A  family $\{\A_i: i\in I\}$ of unital subalgebras of a non-commutative probability space $(\A,\varphi)$ is freely independent (or free) if $\varphi (a_1a_2\cdots a_n)=0$ whenever the following conditions are met: $a_i\in \A_{l(i)}$, $\varphi (a_i)=0$ for $i=1, 2, \cdots, n$, and $l(i)\ne l(i+1)$, for $i=1, 2, \cdots, n-1$. A family $\{a_i:i\in I\}$ of elements is free if the unital subalgebras generated by $a_i$'s are free.

 {\bf Non-crossing partitions.} Given a natural number $m\ge 1$, let $[m]=\{1, 2, \cdots, m\}$. A partition $\pi$ of $[m]$ is a collection of non-empty disjoint subsets of $[m]$ such that the union of all subsets in $\pi$ is $[m]$. A partition $\pi=\{B_1, B_2,\cdots, B_r\}$ of $[m]$ is non-crossing if one cannot find two block $B_i$ and $B_j$ of $\pi$, and four numbers $p_1, p_2\in B_i$, $q_1, q_2\in B_j$ such that $p_1<q_1<p_2<q_2$. The collection of all non-crossing partitions of $[m]$ is denoted by $NC(m)$. $|NC(m)|$, the number of non-crossing partitions of $[m]$,  is $C_m=\frac{(2m)!}{m!(m+1)!}$, which is called the $m$-th Catalan number (Notation 2.9 in \cite{NS}).

 {\bf The Mobius function.} Let $P$ be a finite partial ordered set (poset), and $P^{(2)}=\{(\pi, \sigma): \pi, \sigma\in P, \pi \le \sigma\}$. For two functions $F, G:P^{(2)}\rightarrow \mathbb{C}$, we define the convolution $F*G$ by $$F*G(\pi, \sigma):=\sum_{\rho\in P, \pi\le \rho\le \sigma}F(\pi, \rho)G(\rho, \sigma).$$ Let $\delta(\pi,\sigma)=1$, if $\pi=\sigma$; $\delta(\pi, \sigma)=0,$ if $\pi <\sigma$. Then
 $$F*\delta(\pi,\sigma)=\sum_{\rho\in P, \pi\le \rho\le \sigma}F(\pi, \rho)\delta(\rho, \sigma)=F(\pi, \sigma), \forall F.$$ It follows that $\delta$ is the unit of set of all functions on $P^{(2)}$ with respect to convolution $*$.  The inverse function of the function $\zeta: P^{(2)}\rightarrow \mathbb{C}$, $\zeta(\pi, \sigma)=1, \forall (\pi, \sigma)\in P^{(2)}$, with respect to the convolution $*$ is called the Mobius function $\mu_P$ of $P$.

 {\bf Free Cumulants} Let $\pi, \sigma\in NC(n)$. We say $\pi\le \sigma$ if each block (a subset of $[n]$) of $\pi$ is completely contained in one of the blocks of $\sigma$. $NC(n)$ is a poset by this partial order. The Mobius function of $NC(n)$ is denoted by $\mu_n$. The unital linear functional $\varphi:\A\rightarrow \mathbb{C}$ produces a sequence of multilinear functionals $$\varphi_n:\A^n\rightarrow \mathbb{C}, \varphi_n(a_1, a_2, \cdots, a_n)=\varphi (a_1a_2\cdots a_n), n=1, 2, \cdots.$$ Let $V=\{i_1, i_2, \cdots, i_s\}\subseteq [n]$. We define $\varphi_V(a_1, a_2, \cdots, a_n)=\varphi (a_{i_1}a_{i_2}\cdots a_{i_s})$. More generally, for a partition $\pi=\{V_1, V_2, \cdots, V_r\}\in NC(n)$, we define $\varphi_\pi(a_1, a_2, \cdots, a_n)=\prod_{i=1}^r\varphi_{V_i}(a_1, a_2, \cdots, a_n)$. The $n$-th free cumulant of $(\A,\varphi)$ is the multilinear functional $\kappa_n:\A^n\rightarrow \mathbb{C}$ defined by   $$\kappa_n(a_1, a_2, \cdots, a_n)=\sum_{\pi\in NC(n)}\varphi_\pi (a_1, a_2, \cdots, a_n)\mu_n(\pi, 1_n), $$ where $1_n=[n]$ is the single-block partition of $[n]$.

 Free cumulants $\kappa_n:\A^n\rightarrow \mathbb{C}$ and free independence have a very beautiful relation.
 \begin{Theorem}[Theorem 11.20 in \cite{NS}] A family $\{a_i:i\in I\}$ of elements in $(\A, \varphi)$ is freely independent if and only if for all $n\ge 2$ and all $i(1), i(2), \cdots, i(n)\in I$, $$\kappa_n(a_{i(1)}a_{i(2)}\cdots a_{i(n)})=0$$ whenever there exist $1\le l, k\le n$ with $i(l)\ne i(k)$. Therefore, if $a$ and b are freely independent, then $\kappa_n(a+b)=\kappa_n(a+b, a+b, \cdots, a+b)=k_n(a)+k_n(b)$.
 \end{Theorem}

 {\bf Semicircle elements.} Let $(\A,\varphi)$ be a $*$-probability space. A self-adjoint element $a\in \A$ is a semicircle element (or has a semicircle distribution) if $$\varphi (a^n)=\frac{2}{\pi r^2}\int_{-r}^rt^n\sqrt{r^2-t^2}dt, n=1, 2, \cdots,$$ where $r$ is called the radius of $a$. When $r=2$, $\varphi (a^2)=1$, we say $a$ a standard semicircle element (or has a standard semicircle distribution). A semicircle element can be characterized by $\varphi(a^{2k})=(r^2/4)^kC_k$, where $C_k$ is the $k$-th Catalan number, and $\varphi (a^{2k+1})=0$, $k=0, 1, 2, \cdots$, or by free cumulants $\kappa_n(a)=\delta_{n,2}\frac{r^2}{4}$ ((11.13) in \cite{NS}).

\section{Multidimensional free Poisson distributions}

By the discussion in Page 203 and Exercise 12.22 of \cite{NS}, a classical Poisson distribution is the limit in distribution of a sequence of convolutions of Bernoulli distributions. In the point of view of random variables, we can restate it as follows. Let $\lambda >0, \alpha\in \mathbb{R}$. For each $N\in \mathbb{N}, N>\lambda$, let $\{b_{i,N}:i=1,2, \cdots, N\}$ be a sequence of i.i.d. Bernoulli random variables such that $$Pr(b_{i,N}=0)=1-\frac{\lambda}{N}, Pr(b_{i,N}=\alpha)=\frac{\lambda}{N}.$$ Then the binomial random variable $S_N=\sum_{i=1}^Nb_{i,N}$ has a binomial distribution $$Pr(S_N=k\alpha)=C_N^k(\frac{\lambda}{N})^k(1-\frac{\lambda}{N})^{N-k},$$ $k=0, 1, 2, \cdots, N$,  where $C_N^k$ is the combination number (or the binomial coefficient). Let $N\rightarrow \infty$, by elementary calculus, we can get
$$\lim_{N\rightarrow \infty}Pr(S_N=k\alpha)=\frac{\lambda^k}{k!}e^{-\lambda}=Pr(P=k\alpha),$$ where $P$ has a Poisson distribution $Pr(P=k\alpha)=\frac{\lambda^k}{k!}e^{-\lambda}$, $k=0, 1, 2, \cdots$.

In non-commutative case, the free Poisson limit theorem (Proposition 12.11 in \cite{NS}) says that a free Poisson distribution is the limit in distribution of a sequence of free convolutions of Bernoulli distributions. We want to restate it in the language of random variables.

 Let's define Bernoulli  random variables in a non-commutative probability space. Let $(\A,\varphi)$ be a non-commutative probability space.  A {\sl Bernoulli random variable} $a\in \A$ is a linear combination $a=\alpha p+\beta (1-p)$, where $\alpha,\beta \in \mathbb{R}$, and $p\in \A$ is an idempotent ($p^2=p$) with $0\le \varphi(p)\le 1$. The classical interpretation of a Bernoulli random variable is that $a$ is a random variable with two ``values": $\alpha$ and $\beta$, and $Pr(a=\alpha)=\varphi(p), Pr(a=\beta)=1-\varphi(p)$. In the free Poisson limit theorem,  $\beta=0$, $\varphi(p)=\frac{\lambda}{N}, N>\lambda$. We can restate the free Poisson limit theorem as follows. Let $\lambda>0, \alpha\in \mathbb{R}$. For $N\in \mathbb{N}, N>\lambda$, let $\{\alpha p_{1,N}, \alpha p_{2,N}, \cdots, \alpha p_{N,N}\}$ be a free family of Bernoulli random variables such that $\varphi(p_{i,N})=\frac{\lambda}{N}, i=1, 2, \cdots, N$. Let $S_N=\sum_{i=1}^N\alpha p_{i,N}$. Then
 $$\lim_{N\rightarrow \infty}\kappa_m(S_N)=\lambda \alpha^m, m=1,2, \cdots.$$
Hence, we may restate the definition of free Poisson random variables as follows.
\begin{Definition}[Proposition 12.11, Definition 12.12 \cite{NS}] Let $\lambda\ge 0, \alpha \in \mathbb{R}$, and $(\A,\varphi)$ a non-commutative probability space. A random variable $a\in \A$  has a free Poisson distribution if the free cumulants of $a$ are $\kappa_n(a)=\lambda \alpha^n, \forall n\in \mathbb{N}$.
\end{Definition}

In this section, we want to generalize the results on free Poisson distributions in Lecture 12 of \cite{NS} to the multidimensional case.

By the proof of Theorem 13.1  in \cite{NS}, we can modify the theorem slightly to be the following form.
\begin{Proposition}[Theorem 13.1 and Lemma 13.2 in \cite{NS}]
Let $\{n_k\}$ be a sequence of natural numbers such that
$\lim_{k\rightarrow \infty}n_k =\infty$, and, for each natural
number $k$, $(\A_k, \varphi_k)$ be a non-commutative probability
space. Let $I$ be an index set. Consider a triangular array of
random variables, i. e., for each $i\in I$, $0\le r\le n_k$, we have
a random variable $a^{(i)}_{n_k, r}\in \A_k$. Assume that, for each
$k$, the sets $\{a^{(i)}_{(n_k,1)}\}_{i\in
I},\{a^{(i)}_{(n_k,2)}\}_{i\in I}, \cdots,
\{a^{(i)}_{(n_k,n_k)}\}_{i\in I} $ are free and identically
distributed. Then the following statements are equivalent.
\begin{enumerate}
\item There is a family of random variables $(b_i)_{i\in I}$ in some non-commutative probability space $(\A,\varphi)$ such that
$(a^{(i)}_{n_k,1}+a^{(i)}_{n_k,2}+\cdots+a^{(i)}_{n_k,n_k})_{i\in I}$ converges in distribution to $(b_i)_{i\in I}$, as $k\rightarrow \infty$.
\item For all $n\ge 1$, and all $i(1), i(2), \cdots, i(n)\in I$, the limits
$\lim_{k\rightarrow \infty}n_k\varphi_k(a^{(i(1))}_{n_k,r}\cdots a^{(i(n))}_{n_k,r})$ exist, $1\le r\le n_k$.
\item For all $n\ge 1$, and all $i(1), i(2), \cdots, i(n)\in I$, the limits
$\lim_{k\rightarrow \infty}n_k \kappa_n^k(a^{(i(1))}_{n_k,r}\cdots a^{(i(n))}_{n_k,r})$ exist, $1\le r\le n_k,$ where $\kappa_n^k$ is the $n$-th free cumulant functional in $\A_k$.
\end{enumerate}
Furthermore, if one  of these conditions is satisfied, then the
limits in $(2)$ are equal to the corresponding limits in $(3)$, and
the joint distribution of the limit family $(b_i)_{i\in I}$ is
determined in terms of free cumulants by ($n\ge 1, i(1), i(2),
\cdots, i(n)\in I$)
$$\kappa_n (b_{i(1)}b_{i(2)}\cdots b_{i(n)})=\lim_{k\rightarrow \infty}n_k \varphi_k (a^{(i(1))}_{n_k,r}a^{(i(2))}_{n_k,r}\cdots a^{(i(n))}_{n_k,r}). $$
\end{Proposition}

We will use the following elementary result in sequel.
\begin{Lemma}Let $\{a_{i,j}: i, j=1, 2, \cdots\}$ be a bi-index sequence of complex numbers. If $\sup\{|a_{i,j}|:i=1,2,\cdots\}=M_j<\infty, \forall j$, then there exists a sequence $(n_k)_{k\in \mathbb{N}}$ of natural numbers such that $\lim_{k\rightarrow \infty}n_k=\infty$, and $\lim_{k\rightarrow \infty}a_{n(k),j}$ exists,$\forall j\in \mathbb{N}$.
\end{Lemma}
\begin{proof}
Since $\{|a_{i,1}|:i=1,2,\cdots\}$ is bounded, there is a sequence $\{i(k,1):k=1,2,\cdots\}$ of natural numbers such that $\lim_{k\rightarrow \infty}a_{i(k,1),1}=a_1$, for some number $a_1$. Consider the sequence $\{a_{i(k,1),2}:k=1,2,\cdots\}$. Since the sequence is bounded, there is a subsequence $\{i(k,2):k=1,2,\cdots\}$ of $\{i(k,1):k=1,2,\cdots\}$ such that $\lim_{k\rightarrow \infty}a_{i(k,2),2}=a_2$. But we also have $\lim_{k\rightarrow\infty}a_{i(k,2),1}=a_1$. Continuing the process, we can obtain a bi-index sequence $\{i(k,l):k,l=1,2,\cdots\}$ of natural numbers such that $\lim_{k\rightarrow \infty}a_{i(k,l),j}=a_j$, for $j\le l$. Let $n_k=i(k,k)$, for $k=1,2,\cdots$. Then, for a $j$, and an $\epsilon >0$, there there exists a natural number $K>j$ such that $|a_{i(k,j),j}-a_j|<\epsilon,\forall k>K$. Note that $\{i(n,k); n=1,2,\cdots\}$ is a subsequence of $\{i(n,j):n=1,2,\cdots\}$. Thus, $i(k,k)\ge i(k,j)$, and $|a_{i(k,k)}-a_j|<\epsilon,\forall k>K$. It means that $\lim_{k\rightarrow \infty}a_{n_k,j}=a_j, \forall j$.
 \end{proof}
 \begin{Theorem} Let $\{\alpha_i:i=1,2,\cdots\}$ be a sequence of real numbers, $\{\lambda_i\ge 0\}_{i\in\mathbb{N}}$ with $\lambda=\sup\{\lambda_i:i\ge 1\}<\infty$,  and for each $N\in \mathbb{N}, N>\lambda$, there be $N$ freely independent and identically distributed sequences $$\{p_{1,N}^{(i)}\}_{i\in \mathbb{N}}, \{p_{2,N}^{(i)}\}_{i\in \mathbb{N}}, \cdots, \{p_{N,N}^{(i)}\}_{i\in \mathbb{N}}$$ of commutative projections on a $C^*$-probability space $(\A_N,\varphi_N)$, i. e., $p_{j,N}^{(i(1))}p_{j,N}^{(i(2))}=p_{j,N}^{(i(2))}p_{j,N}^{(i(1))}$, $\forall i(1), i(2)=1,2,\cdots$. Moreover, $\varphi_N(p^{(i)}_{r,N})=\frac{\lambda_i}{N}, i=1,2,\cdots, r=1,2,\cdots, N$.  Define a triangular family of sequences of random variables $\{a_{j, N}^{(i)}=\alpha_ip^{(i)}_{j,N}:i=1, 2, \cdots\}$, for $ j=1, 2, \cdots, N, N=1,2,\cdots.$ Then there exists a family of random variables $(b_i)_{i\in\mathbb{N}}$ in a non-commutative probability space $(\A,\varphi)$ and a sequence $\{n_k:k=1,2,\cdots\}$ of natural numbers such that $\lim_{k\rightarrow \infty}n_k=\infty$ and  $(a_{1, n_k}^{(i)}+a_{2,n_k}^{(i)}+\cdots +a_{n_k, n_k}^{(i)})_{i\in \mathbb{N}}$ converges to $(b_i)_{i\in \mathbb{N}}$ in distribution, as $k\rightarrow \infty$.
 \end{Theorem}
\begin{proof}
For $N, i(1), i(2), \cdots, i(n), n\in \mathbb{N}$, let $$f(N, i(1), i(2), \cdots, i(n))=N\varphi_N(a^{(i(1))}_{r,N}a^{(i(2))}_{r,N}\cdots a^{(i(n))}_{r,N}), 1\le r\le N, $$
and $M(i(1), i(2),\cdots, i(n))= |\alpha_{i(1)}\alpha_{I(2)}\cdots\alpha_{i(n)}| \min\{\lambda_{i(j)}:j=1,2,\cdots, n\}.$ Then
\begin{align*}
&|f(N,i(1), i(2), \cdots, i(n))|=N|\alpha_{i(1)}\alpha_{i(2)}\cdots\alpha_{i(n)}\varphi_N(p^{(i(1))}_{r,N}p^{(i(2))}_{r,N}\cdots p^{(i(n))}_{r,N})|\\
\le & N |\alpha_{i(1)}\alpha_{i(2)}\cdots\alpha_{i(n)}|\min\{|\varphi_N(p^{(i(j))}_{r,N})|:j=1,2,\cdots n\}\\
=& N|\alpha_{i(1)}\alpha_{i(2)}\cdots\alpha_{i(n)}| \min\{\frac{\lambda_{i(j)}}{N}:j=1,2,\cdots, n\}\\
=&M(i(1), i(2),\cdots, i(n)),
 \end{align*}
 since $\varphi_N:A_N\rightarrow \mathbb{C}$ is positive, and $p^{(i(1))}_{r,N}p^{(i(2))}_{r,N}\cdots p^{(i(n))}_{r,N}\le \min\{p^{(i(j))}:j=1,2,\cdots, n\}$, as projections.

 Let $$S_m=\{(i(1), i(2), \cdots, i(n)): i(1)+i(2)+\cdots +i(n)=m, i(1), i(2), \cdots, i(n)\in \mathbb{N}\},$$ for $m\in \mathbb{N}$. Then for each $ m\in \mathbb{N}$, $S_m$ is a finite set with $|S_m|=k_m$, and $\{S_m:m\in \mathbb{N}\}$ is a partition of the set $\{(i(1), i(2), \cdots, i(n): i(1), i(2), \cdots, i(n), n\in \mathbb{N}\}$. Define a bijective map
 $$\gamma: S_1\rightarrow \{1\}, \gamma: S_m\rightarrow \{(\sum_{l=1}^
 {m-1}k_{l})+1, (\sum_{l=1}^{m-1}k_{l})+2, \cdots, \sum_{l=1}^
 {m}k_{l}\}, m\ge 2. $$ For instance,   $\gamma((1,1))=2, \gamma(2)=3, \gamma(S_2)=\{2,3\}$.  It implies that
\begin{align*}
&\gamma(\{(i(1), i(2), \cdots, i(n)): i(1), i(2), \cdots, i(n), n\in \mathbb{N}\})\\
=&\gamma(S_1)\cup \gamma(S_2)\cup\cdots\cup \gamma(S_m)\cup\cdots\\
=&\{j:j=1,2, 3, \cdots\}.
\end{align*}
Thus, $\{f(N,i(1), i(2), \cdots, i(n)):N, i(1), i(2), \cdots, i(n),
n\in \mathbb{N}\}=\{f(N,\gamma^{-1}(j)):N, j\in\mathbb{N}\}$
 is a bi-index sequence. By Lemma 2.3, there is a sequence $(n_k)_{k\in\mathbb{N}}$ of natural numbers such that $\lim_{k\rightarrow \infty} n_k=\infty$, and $f(n_k, i(1), i(2), \cdots, i(n))$ converges  as $k\rightarrow\infty$, for every tuple $(i(1), i(2), \cdots, i(n))$. By Theorem 2.2, there is a family of random variables $(b_i)_{i\in \mathbb{N}}$ in a non-commutative probability space $(\A,\varphi)$ such that
$((a^{(i)}_{1,n_k}+a^{(i)}_{2,n_k}+\cdots+a^{(i)}_{n_k,n_k})_{i\in \mathbb{N}}$ converges to $(b_i)_{i\in \mathbb{N}}$ in distribution.
\end{proof}
\begin{Remark}

\begin{enumerate}
\item By Theorems 2.2 and 2.4, For each $i\in \mathbb{N}$,
$$\kappa_n(b_i)=\lim_{k\rightarrow\infty}(n_k\varphi_{n_k}((a^{(i)}_{r,n_k})^n)=\alpha_i^n\lambda_i, n=1, 2, \cdots.$$
Hence, $b_i$ has a free Poisson distribution, for each $i\in \mathbb{N}$.
\item If $\{p^{(i)}_{r,N}\}_{i\in \mathbb{N}}$ is an orthogonal sequence of  projections, for $\forall N, r=1,2,\cdots, N$, then
$$\kappa_n(b_{i(1)}b_{i(2)}\cdots b_{i(n)})=\lim_{k\rightarrow \infty}n_k\varphi_{n_k}(\alpha_{i(1)}\cdots\alpha_{i(n)}p^{(i(1))}_{r,n_k}\cdots p^{(i(n))}_{ r,n_k})=0,$$ whenever there are $i(j)\ne i(l), 0\le j,l\le n$. This means that $\{b_i:i\in \mathbb{N}\}$ is a free family of free Poisson random variables. A similar procedure of constructing a free family from an orthogonal one can be found in Example 12.19 in \cite{NS}.
\end{enumerate}
\end{Remark}
We, therefore, define multidimensional free Poisson distributions as follows.
\begin{Definition} A family of random variables $(b_i)_{i\in \mathbb{N}}$ in a non-commutative probability space $(\A,\varphi)$ has a joint free Poisson distribution if the family has a joint distribution same as the limit distribution in Theorem 2.4.
\end{Definition}

\section{Free Poisson Processes}
 An analogue of classical Poisson processes in free probability can defined as follows.
 \begin{Definition} For $k\in \mathbb{N}$ and $\alpha\in \mathbb{R}$. A family $\{X_t:t\ge 0\}$ of self-adjoint elements in a $*$-non-commutative probability space $(\mathcal{A}, \varphi)$
 is a free Poisson process if it satisfies the following conditions.
 \begin{enumerate}
 \item $X_0=0$.
 \item For $0\le t_1<t_2<\cdots <t_n<\infty$, $X_{t_n}-X_{t_{n-1}}, \cdots, X_{t_2}-X_{t_1}$ form a freely independent family.
 \item For $0\le s< t$, $X_t-X_s$ has a free Poisson distribution with parameters $\lambda=k(t-s)$ and $\alpha$, that is, $\kappa_n(X_t-X_s)=k(t-s)\alpha^n, n=1,2,\cdots.$  (The most common case is that $k=1$.)
 \end{enumerate}
 \end{Definition}
 {\bf Construction} We give a procedure for constructing a free Poisson process in a  $*$-non-commutative probability space, for a real number $\alpha$.
\begin{enumerate}
\item For each natural number $N$, let $t\mapsto p_{t,N}$ be a projection-values process $[0,N]\rightarrow \mathcal{A}$, where $(\A,\varphi)$ be a $W^*$-probability space. That is, for $0\le t_1<t_2<\cdots <t_n\le N$,
$\{p_{t_2,N}-p_{t_1,N}, p_{t_3,N}-p_{t_2,N},\cdots,
p_{t_n,N}-p_{t_{n-1},N}\}$ is an orthogonal family of projections, and
$\varphi (p_{t,N})=\frac{t}{N},\forall 0\le t\le N$. Actually, we
can get such a process by letting
$p_t-p_s=p_{[\frac{s}{N},\frac{t}{N})}$, for $0\le s\le t\le N$ in
the projection-valued process in Section 4.2 of \cite{MA}.
\item Choose $N$ free copies of the process in Step 1. That is,  processes $\{p_{t,N}^i:0\le t\le N\}, i=1,2, \cdots, N$, are $N$ free families of random variables, and   $\varphi(p_{t,N}^i)=\frac{t}{N}$ $i=1,2,\cdots, N$.
\item Let $a_{t,N}=\alpha \sum_{i=1}^Np_{t,N}^i$. Then the limit in distribution of $a_{t,N}$, as $N\rightarrow \infty$,  has free cumulants $\kappa_n(a_t)=t\alpha^n,\forall n\ge 1$, by Proposition 12.11 in \cite{NS}.
\item By Exercise 16.21 or Theorem 21.7 in \cite{NS}, there is a family $\{a_t:t\ge 0\}$ in $\mathcal{A}$ such that $\kappa_n(a_t)=t\alpha^n, \forall n\ge 1$
 (If necessary, we can expand $\mathcal{A}$ so that $\mathcal{A}$ contains all limit elements $\{a_t, t\ge 0\}$).
\end{enumerate}
Now we show that $\{a_t:t\ge 0\}$ is a free Poisson process with parameter $\alpha$.
\begin{Theorem} The process $\{a_t:t\ge 0\}$ constructed via the above procedure is a free Poisson process in $(\mathcal{A},\varphi)$.
\end{Theorem}
\begin{proof} For $0\le s<t$, choose $N>t$, and consider $a_{t,N}-a_{s,N}=\alpha\sum_{i=1}^N(p_{t,N}^i-p_{s,N}^i)$. By the proof of Proposition 12.11 in \cite{NS}, the $n$-th free cumulant of
$a_{t,N}-a_{s,N}$ is $(t-s)\alpha^n+O(\frac{1}{N})$. Thus,
 $$\kappa_n(a_t-a_s)=\lim_{N\rightarrow \infty}((t-s)\alpha^n+O(\frac{1}{N}))=(t-s)\alpha^n,\forall n\ge 1.$$
Moreover, by Remark 2.5, for $0\le t_1<t_2<\cdots <t_n<\infty$, $a_{t_n}-a_{t_{n-1}}, \cdots a_{t_2}-a_{t_1}$ form a freely independent family.
\end{proof}
\begin{Remark}
\begin{enumerate}
\item  In Section 4.2 of \cite{MA}, Anshelevich constructed a free Poisson process as follows. For a projection-valued process $I\mapsto p_I$, from
 half-open intervals $I\subset [0,1]$ into projections in a $W^*$-probability space$(\mathcal{A}, \varphi)$, and a standard semicircle element $s\in \mathcal{A}$, which is free from $\{p_I:I\subset [0,1]\}$, $I\mapsto sp_Is$ is a free Poisson process with $\kappa_n(sp_Is)=|I|$. Let $I=[0,t)$. Then $\kappa_n(sp_Is)=t$ ($t>0$). This is a special case of free Poisson processes with $\alpha=1$.
 \item When combining  the construction in Remark 1.9 in \cite{NS1} and that in Section 4.2 in \cite{MA}, we can get a free Poisson process $I\mapsto sp_Is$ for a general semicircle element $s\in \mathcal{A}$ with radius $r$. The free Poisson process $I\mapsto sp_Is$ has $n$-th cumulant $\kappa_n(sp_Is)=|I|\frac{r^{2n}}{4^n}$. This is a free Poisson process with $\alpha=\frac{r^2}{4}>0$.
 \item One can get a free Poisson process by our procedure for any $\alpha\in \mathbb{R}$. \end{enumerate}
\end{Remark}

A generalized version of free Poisson distributions is the following {\sl compound free Poisson distributions}.
\begin{Definition}[Definition 12.16 in \cite{NS}]
Let $\nu$ be a compactly supported probability measure on $\mathbb{R}$, and $\lambda\ge 0$. A probability measure $\mu$ on $\mathbb{R}$ is called a {\sl{\bf compound free Poisson distribution}} with rate $\lambda$ and jump distribution $\nu$ if the $n$-th free cumulant of $\mu$ is $\kappa_n(\mu)=\lambda m_n(\nu)$, where $m_n(\nu)$ is the $n$-th moment of measure $\nu$, and $n\ge 1$.
\end{Definition}
Now we generalize the notion of free Poisson processes to a {\sl compound} version.

\begin{Definition} Let $\nu$ be a compactly supported probability measure on $\mathbb{R}$ and $k\in \mathbb{N}$. A family $\{a_I:I=[s,t)\subset [0,\infty)\}$ of self-adjoint elements in a $*$-non-commutative probability space $(\mathcal{A},\varphi)$ is a compound free Poisson process with respect to  $\nu$, if it satisfies the following conditions.
\begin{enumerate}
\item $a_I$ has a compound free Poisson distribution: $\kappa_n(a_I)=k|I|m_n(\nu), n\ge 1$.
\item If $I_1, I_2, \cdots, I_n$ are mutually disjoint half-open intervals, then $a_{I_1}, a_{I_2}, \cdots, a_{I_n}$ form a freely independent family.
\end{enumerate}
\end{Definition}
{\bf Construction} For a compactly supported probability measure $\nu$ on $\mathbb{R}$, the construction for a compound free Poisson process is very similar to that for a free Poisson process.
\begin{enumerate}
\item For each natural number $N$, choose a projection-valued process $I\mapsto p_{I,N}$, from half-open intervals $I\subset [0,N]$ into a $W^*$-probability space $(\mathcal{A},\varphi)$ such that $\varphi(p_{I,N})=\frac{|I|}{N}$, as we did in the previous construction for a free Poisson process. Let $p_{t,N}=p_{[0,t), N}$.
\item For a half-open interval $I=[s,t)$, where $0\le s<t$, choose a natural number $N>t$, and a self-adjoint  element $a_{I,N}$ in
$(\mathcal{A}_I,\varphi_I)$ such that $\varphi_I (a_{I,N}^n)=m_n(\nu)$, where $\mathcal{A}_I=p_{I,N}\mathcal{A}p_{I,N}$, $\varphi_I(x)=\frac{\varphi(x)}{\varphi(p_{I,N})}$, for $x\in \mathcal{A}_I$, and $m_n(\nu)$ is the $n$-th moment of measure $\nu$. Then $\varphi(a_{I,N}^n)=\varphi(p_{I,N})m_n(\nu)=\frac{t-s}{N}m_n(\nu), n\ge 1$.
\item As the Step 2 in the previous construction, for each natural number $N$, choose $N$ processes $\{a_{I,N}^{(i)}:I=[s,t)\subset [0,N]\}, i=1,2, \cdots, N$,
 that is, $\varphi((a_{I,N}^{(i)})^n)=\frac{|I|}{N}m_n(\nu), n\ge 1$,  $i=1,2,\cdots, N$. Also, the $N$ processes are freely independent from each other.
\item Let $b_{I,N}=\sum_{i=1}^Na_{I,N}^{(i)}$, $I\subset [0,N]$, and $N=1, 2, \cdots$. By the proof of Proposition 12.11 in \cite{NS},
 $\lim_{N\rightarrow \infty}\kappa_n(b_{I,N})=|I|m_n(\nu), n\ge 1$.
 \item Let $\{b_I:I=[s,t), 0\le s<t<\infty\}$ be the family of random variables with the distribution $$\kappa_n(b_t)=(t-s)m_n(\nu), n\ge 1.$$
\end{enumerate}
\begin{Theorem} $\{b_I:I=[s,t), 0\le s<t\}$ constructed above is a compound free Poisson process.
\end{Theorem}
\begin{proof}
By the construction, $\kappa_n(b_I)=|I|m_n(\nu), n\ge 1$, that is, $b_I$ has a compound free Poisson distribution.

For a family $\{I_1, I_2, \cdots, I_k\}$ of mutually disjoint
half-open intervals, choose a natural number $N$ such that
$I_j\subset [0,N], j=1, 2, \cdots, k$. For every $0<r\le N$,
$a_{I_j,N}^{(r)}\in
\mathcal{A}_{I_j,N}=p_{I_j,N}\mathcal{A}p_{I_j,N}$, $j=1, 2, \cdots,
N$. Thus, $\prod_{j=1}^ka_{I_j,N}^{(r)}=0$. By Theorem 13.1 in
\cite{NS},
$$\kappa_l(b_{I_{i(1)}},b_{I_{i(2)}},\cdots b_{I_{i(l)}})=\lim_{N\rightarrow\infty}N\varphi (\prod_{j=1}^la_{I_{i(j)},N}^{(r)})=0, $$
 if there are two disjoint intervals in $\{I_{i(1)}, I_{i(2)},\cdots, I_{i(l)} \}\subseteq \{I_1, I_2, \cdots, I_k\}, \forall 1< l\le k$.
It follows from Theorem 11.20 in \cite{NS} that $b_{I_1}, b_{I_2}, \cdots, b_{I_k}$ form a freely independent family.
\end{proof}

In classical probability, the sum of two independent Poisson processes is still a Poisson process (see Section 2.3 in \cite{RG}). In free probability, we have a slightly different result.
\begin{Theorem} Let $\{b_{I,i}:I=[s,t)\subset [0,\infty)\}, i=1,2,\cdots, k$, be a freely independent family of $k$ compound free Poisson processes
such that $\kappa_n(b_{I,i})=|I|m_n(\nu_i)$, where $\nu_i$ is a
compactly supported probability measure on $\mathbb{R}$, $i=1, 2,
\cdots, k$. Then the sum $b_I=\sum_{i=1}^kb_{I,i}$ is a compound
free Poisson process with distribution $\kappa_n(b_I)=k|I|m_n(\nu),
n\ge 1$, where $\nu=\sum_{i=1}^k\frac{\nu_i}{k}$.
\end{Theorem}
\begin{proof}
For a half-open interval $I=[s,t)$, operators $b_{I,1}, b_{I,2}, \cdots, b_{I,k}$ are freely independent. Therefore,
$$\kappa_n(b_I)=\sum_{i=1}^k\kappa_n(b_{I,i})=\sum_{i=1}^k|I|m_n(\nu_i)=k|I|\sum_{i=1}^k\frac{m_n(\nu_i)}{k}=k|I|m_n(\nu), n\ge 1.$$
For mutually disjoint intervals $I_1, I_2, \cdots, I_l$ of $[0,\infty)$ and $n\le l$, we have
$$\kappa_n(b_{I_1},b_{I_2},\cdots, b_{I_n})=\sum_{i_1,i_2, \cdots, i_n=1;i_p\ne i_{p'},\forall p\ne p'}^n\sum_{j_1,j_2,\cdots, j_n=1}^k\kappa_n(b_{I_{i_1},j_1}, b_{I_{i_2},j_2},\cdots, b_{I_{i_n}, j_n}).$$
It implies that $b_{I_1}, b_{I_2}, \cdots, b_{I_l}$ are freely independent.
\end{proof}
\begin{Corollary} Let $\{a_{t,1}:t\ge 0\}$ and $\{a_{t,2}:t\ge 0\}$ be two freely independent free Poisson processes with distributions $\kappa_n(a_{t,i})=t\alpha_i^n, n\ge 1, i=1,2$. Then $\{a_t=a_{t,1}+a_{t,2}:t\ge 0\}$ is compound free Poisson process. Moreover, the sum process
$\{a_t:t\ge 0\}$ of two non-zero free Poisson processes (that is, $(\alpha_1)^2 +(\alpha_2)^2\ne 0$) is a free Poisson process if and only if $\alpha_1=\alpha_2$.
\end{Corollary}
\begin{proof}
The first conclusion that the sum of two freely independent free Poisson processes is a compound free Poisson process follows from the previous result, Theorem 3.7. More precisely, the sum process has a distribution $\kappa_n(a_t)=2tm_n(\nu), n\ge 1$, where $\nu=\frac{1}{2}(\delta_{\alpha_1}+\delta_{\alpha_2})$, and $\delta_\alpha$ is the point mass distribution at number $\alpha$ (see Example 3.4.1 in \cite{VDN}). Note that a compound free Poisson distribution is a  free Poisson distribution if and only if the probability measure $\nu$ in the definition of compound free Poisson distributions is a point mass distribution $\delta_\alpha$. If $\alpha_1=\alpha_2$, then $\delta_{\alpha_1}+\delta_{\alpha_2}=2\delta_\alpha$, where $\alpha=\alpha_1=\alpha_2$. Therefore, $$\kappa_n(a_{t,1}+a_{t,2})=t(\alpha_1^n+\alpha_2^n)=2t\alpha^n,\forall t\ge0, n\ge 1.$$ It follows that the sum process is a free Poisson process.

 Suppose now that  $\delta_{\alpha_1}+\delta_{\alpha_2}=2\delta_{\alpha}$, for some $\alpha\in \mathbb{R}$, we will show that $\alpha_1=\alpha_2$.
 Suppose $\alpha_1\ne \alpha_2$.  If $|\alpha_1|>|\alpha_2|$, then for $t>0$ and all $n\in \mathbb{N}$, we have
$$\kappa_n(a_{t,1}+a_{t,2})=t(\alpha_1^n+\alpha_2^n)=2t\alpha^n\Rightarrow 1+(\frac{\alpha_2}{\alpha_1})^n=2(\frac{\alpha}{\alpha_1})^n.$$
Let $n\rightarrow \infty$, we have the limit of the left side
sequence is $1$. On the other hand, if $|\alpha|>|\alpha_1|$, the
limit of the right side sequence is $\pm\infty$; if
$|\alpha|<|\alpha_1|$, the limit of the right side sequence is $0$;
if $\alpha=-\alpha_1$, the limit of the right side sequence $(-1)^n$
does not exist. All these cases lead to a contradiction. It implies
that $\alpha=\alpha_1$. Thus, $\lim_{n\rightarrow
\infty}2(\frac{\alpha}{\alpha_1})^n=2$. This contradicts to the fact
that the limit on the left is 1. Hence, the assumption
$|\alpha_1|>|\alpha_2|$ is wrong. Very similarly, we can prove  that
the assumptions $|\alpha_1|<|\alpha_2|$ and $\alpha_1=-\alpha_2$
lead contradictions. Therefore, if the sum is a free Poisson
process, then $\alpha_1=\alpha_2$.
\end{proof}
\begin{Remark}
\begin{enumerate}
\item A similar result about the sum of free Poisson distributions can be found in Exercise 12.25 in \cite{NS}.
\item Combining the construction after the definition of free Poisson processes and the above corollary, we can construct a free Poisson process $\{a_t:t\ge 0\}$ such that $\kappa_n(a_t)=kt\alpha^n, n\ge 1$, for $k\in \mathbb{N}$ and $\alpha\in \mathbb{R}$.
\end{enumerate}
\end{Remark}
\section{The Karhunen-Loeve expansion of a free Poisson Process}
\subsection{The Karhunen-Loeve expansion of a classical stochastic process}
First let's recall the Karhunen-Loeve Expansion of a classical stochastic process (Sections 2.3.6, 2.3.7 in \cite{DJ} and \cite{AA}).

\begin{Definition}[\cite{PB}]
Let $T\in \mathbb{R}, T>0$. A function $K(s,t):[0,T]\times [0,T]\rightarrow \mathbb{R}$ is called a kernel if
\begin{enumerate}
\item $K$ is symmetric, that is,  $K(s,t)=K(t,s),\forall 0\le s, t\le T$, and
\item $K$ is non-negative definite,   $$\sum_{i,j=1}^nc_ic_jK(t_i, t_j)\ge 0,\forall t_1, t_2, \cdots, t_n\in [0,T], c_1, c_2, \cdots, c_n\in \mathbb{R}.$$
\end{enumerate}
\end{Definition}

\begin{Theorem}[Mercer's theorem, \cite{PB}] Suppose $K:[0,T]\times [0,T]\rightarrow \mathbb{R}$ is a continuous kernel. Then the integral operator $T_K:L^2([0,T])\rightarrow L^2([0,T])$, $T_K(f)(x)=\int_0^TK(x,y)f(y)dy,\forall f\in L^2([0,T])$ is non-negative definite, and there is an orthonormal basis $\{\phi_i:i=1, 2, \cdots\}$ of $L^2([0,T])$ consisting of eigenfunctions of $T_K$ such that the corresponding eigenvalues $\{\lambda_i\}$ are non-negative. The eigenfunctions corresponding to non-zero eigenvalues are continuous on $[0,T]$ and the kernel $K$ has the form $$K(s,t)=\sum_{i=1}^\infty\lambda_i\phi_i(s)\phi_i(t),$$ where the convergence is absolute and uniform, that is, $$\lim_{n\rightarrow \infty}\sup_{s,t\in [0,T]}|K(s,t)-\sum_{i=1}^n\lambda_i\phi_i(s)\phi_i(t)|=0.$$
\end{Theorem}

\begin{Theorem}[Karhunen-Loeve, \cite{AA}] Let $X_t$ be a centered  (i. e. $E(X_t)=0$),  mean square continuous ($\lim_{t\rightarrow t_0}E((X_t-X_{t_0})^2)=0$) stochastic process on a probability space $(\Omega, \mathcal{F}, \mu)$ with $X_t\in L^2(\Omega, [0,T])$. Then there is an orthonormal basis $\{\phi_i\}$ of $L^2([0,T])$ such that for all $t\in [0,T]$, $$X_t=\sum_{i=1}^\infty X_i\phi_i(t),$$ where $X_i=\int_0^TX_t\phi_i(t)dt$, the convergence is in $L^2(\Omega)$, and $E(X_i)=0$, $E(X_iX_j)=\delta_{i,j}\lambda_i, i, j=1, 2, \cdots.$
\end{Theorem}

\subsection{The non-commutative stochastic process case}
Let $\{X(t):t\ge 0\}$ be a free Poisson process in a
$W^*$-probability space $(\A,\varphi)$ with $\alpha\in \mathbb{R},
k=1$ (See Definition 3.1). Suppose that the process is continuous
with respect to   $\|\cdot\|_2$ in $\A$ (which is called
$L^2$-continuous). For $T>0$, let $L^2(0,T;\A)$ be the space of all
$L^2$-continuous non-commutative stochastic processes of
self-adjoint operators in $(\A,\varphi)$. It is obvious that
$(s,t)\mapsto \varphi (X_sX_t)$ is continuous on $[0,T]\times
[0,T]$. Thus, we can define an inner product $\langle
X,Y\rangle=\int_0^T\varphi(X(t)Y(t))dt$ in $L^2(0,T;\A)$. We get a
Hilbert space $L^2(0,T;\A)$.

\begin{Lemma} Let $X_t$ be a free Poisson Process in a $W^*$-probability space $(\A,\varphi)$ with distribution $\kappa_n(X_t)=t\alpha^n, n=1, 2, \cdots, t\ge 0, \alpha\in \mathbb{R}$. Suppose $X_t\in L^2(0,T;\A)$. Then the second free cumulant $k(s,t):=\kappa_2(X_s, X_t)$ is a continuous kernel on $[0,T]$.
\end{Lemma}
\begin{proof} Suppose $0<s\le t\le T$. Then $$k(s,t)=\kappa_2(X_s, X_t)=\kappa_2(X_s, X_t-X_s+X_s)=\kappa_2(X_s,X_s)=s\alpha^2=k(t,s).$$ Moreover, by Example 11.6 in \cite{NS}, $$k(s,t)=\kappa_2(X_s,X_t)=\varphi (X_sX_t)-\varphi (X_s)\varphi (X_t)=\varphi ((X_s-\varphi (X_s))(X_t-\varphi(X_t)))=\varphi (\widetilde{X}_s\widetilde{X}_t),$$ where $\widetilde{X}_t=X_t-\varphi (X_t)$.

It implies that for $0<t_1, t_2, \cdots, t_n\le T$ and $c_1, c_2, \cdots, c_n\in \mathbb{R}$, we have
$$0\le \varphi ((\sum_{i=1}^nc_i\widetilde{X}_{t_i})(\sum_{i=1}^nc_i\widetilde{X}_{t_i}))=\sum_{i,j=1}^nc_ic_j\varphi(\widetilde{X}_{t_i}\widetilde{X}_{t_j})=\sum_{i,j=1}^nc_ic_jk(t_i,t_j).$$ That is, $k$ is non-negative definite.
Note that the continuity of $t\mapsto X_t$ implies the continuity of $t\mapsto \widetilde{X}_t$. Thus, we can assume that $X_t$ is centered, therefore, $k(s,t)=\varphi (X_sX_t)$. It implies that $k(s,t):[0,T]\times [0,T]\rightarrow \mathbb{R}$ is continuous.
\end{proof}
We are in the position to present the  Karhunen-Loeve expansion of a free Poisson Process.
\begin{Theorem}Let$(X_t)_{t\ge 0}$ be an $L^2$-continuous  free Poisson process in a $W^*$-probability space $(\A,\varphi)$ with distribution $\kappa_n(X_t)=t\alpha^n, t\ge 0, n=1, 2, \cdots$, for some $\alpha\in \mathbb{R}$. Let  $$\widetilde{X}_t=X_t-t\alpha I( t\ge 0),  \lambda_n=\frac{\alpha^2T^2}{(n-1/2)^2\pi^2},   \phi_n(t)=(\frac{2}{T})^{1/2}\sin(\frac{(n-1/2)\pi t}{T}), n\ge 1.$$ Then we have the Karhunen-Loeve expansion of $\widetilde{X}_t$.
\begin{enumerate}
\item
$$\widetilde{X}_t=\sum_{i=1}^\infty \widetilde{X}_i\phi_i(t), 0<t\le T,$$ where $\widetilde{X}_i=\int_0^T\widetilde{X}_t\phi_i(t)dt$, the convergence is in $L^2(\A,\varphi)$.
\item $\varphi (\widetilde{X}_i\widetilde{X}_j)=\delta_{i,j}\lambda_i, \forall i, j$.
\end{enumerate}
\end{Theorem}
\begin{proof}
By Lemma 4.4 and Theorem 4.2, there is an orthonormal basis
$\{\phi_i(t):i=1, 2, \cdots\}$ of $L^2([0,T])$ and non-negative
numbers $\lambda_i, i=1, 2, \cdots$. such that $$\varphi
(\widetilde{X}_s\widetilde{X}_t)=\varphi(X_sX_t)-\varphi
(X_s)\varphi (X_t)=k(s,t)=\sum_{i=1}^\infty
\lambda_i\phi_i(s)\phi_i(t), \forall t, s\in [0,T],$$ where the
convergence is absolute and uniform on $[0,T]\times [0,T]$. Let's
prove the properties of $X_i$. $\varphi
(\widetilde{X}_i)=\int_0^T\varphi (\widetilde{X}_t)\phi_i(t)dt=0$,
and
$$\varphi (\widetilde{X}_i\widetilde{X}_j)=\int_0^T\int_0^T\varphi(\widetilde{X}_t\widetilde{X}_s)\phi_i(s)\phi_j(t)dsdt=\int_0^T\int_0^Tk(s,t)\phi_i(s)\phi_j(t)dsdt=\delta_{i,j}\lambda_i.$$
For $t\in [0,T]$ and $n\in \mathbb{N}$, we have
\begin{align*}
\varphi ((\widetilde{X}_t-\sum_{i=1}^n\widetilde{X}_i\phi_i(t))^2)&=\varphi (\widetilde{X}_t^2)-2\sum_{i=1}^n\varphi (\widetilde{X}_t\widetilde{X}_i)\phi_i(t)+\sum_{i,j=1}^n\varphi (\widetilde{X}_i\widetilde{X}_j)\phi_i(t)\phi_j(t)\\
&=k(t,t)-2\sum_{i=1}^n(\int_0^T\varphi (\widetilde{X}_t\widetilde{X}_s)\phi_i(s)ds\phi_i(t))+\sum_{i,j=1}^n\varphi (\widetilde{X}_i\widetilde{X}_j)\phi_i(t)\phi_j(t)\\
&=k(t,t)-2\sum_{i=1}^n\lambda_i\phi_i(t)^2+\sum_{i=1}^n\lambda_i\phi_i(t)^2\rightarrow 0,
\end{align*}
as $n\rightarrow \infty$, by Theorem 4.2. Moreover, the convergence is uniform on $[0,T]$. By the proof of Lemma 5.4 and the example in Page 27 of \cite{DJ}, we can choose the eigenvalues and eigenfunctions as follows
 $$\lambda_n=\frac{\alpha^2T^2}{(n-1/2)^2\pi^2}, \phi_n(t)=(\frac{2}{T})^{1/2}\sin(\frac{(n-1/2)\pi t}{T}).$$
\end{proof}
\section{Integration with respect to free Poisson Random Measures}
A generalization of free Poisson processes is the following {\sl free Poisson random measures}.
\begin{Definition} Let $\alpha\in \mathbb{R}$. A free Poisson random measure is a map $X$ from the set $\mathcal{B}_0$ of all Borel subsets with finite Lebesgue measure on the real line $\mathbb{R}$ into $\mathcal{A}_{sa}$, the space of all self-adjoint elements of a $*$ non-commutative probability space $(\mathcal{A}, \varphi)$, with the following properties.
\begin{enumerate}
\item $X_E$ has a free Poisson distribution with parameters $\alpha$ and $|E|$, for a set $E\in \mathcal{B}_0$, where $|E|$ is the Lebesgue measure of $E$, that is, the $n$-th free cumulant $\kappa_n(X_E)=|E|\alpha^n$, $n\ge 1$.
\item If $E_1, E_2, \cdots, E_k$ are mutually disjoint, then $X_{E_1}, X_{E_2}, \cdots, X_{E_k}$ are freely independent.
\item If $E_1, E_2, \cdots, E_k$ are mutually disjoint, then $X_{\cup_{i=1}^nE_i}=\sum_{i=1}^nX_{E_i}$.
\end{enumerate}
\end{Definition}
\begin{Remark}
 We restrict the map $X$ to $\{[0,t):0< t\}$, and let $X_0=0$,   then $\{0\}\cup \{X_t=X_{[0,t)}:t> 0\}$ is a free Poisson process. Thus a free Poisson measure can be regarded as a generalization of a free Poisson process.
\end{Remark}

For notational simplicity, we assume from now on that $\alpha=1$ in the definition of free Poisson measures. To study integration with respect to a free Poisson random measure, we need to deal with convergence problems. Thus, from now on, we assume that $(\mathcal{A},\varphi)$ is a $W^*$-probability space. Define $\|A\|_p=(\varphi (|A|^p))^{\frac{1}{p}}$, for $A\in \A$. Let $L^p(\mathcal{A},\varphi)$ be the completion of $\mathcal{A}$ with respect to the norm $\|\cdot\|_p$ for $p\ge 1$.

 Let $s=\sum_{i=1}^kc_i\chi_{E_i}$ be a simple function on $\mathbb{R}$, where $E_1, E_2, \cdots, E_n$ are mutually disjoint Borel sets on $\mathbb{R}$ with finite measures, and $c_1, c_2, \cdots, c_n$ are mutually distinct numbers. Define the integral of $s$ with respect to $X_E$ as 
$$X(s)=\int_\mathbb{R}s(x)X(dx):=\sum_{i=1}^nc_iX_{E_i}.$$
Then $\varphi (X(s))=\sum_{i=1}^nc_i\varphi (X_{E_i})=\sum_{i=1}^nc_i|E_i|=\int_{\mathbb{R}}s(x)dx$, and
\begin{align*}
\varphi(X^*(s)X(s))&=\sum_{i,j=1}^nc_i\overline{c_j}\varphi(X_{E_i}X_{E_j})\\
=&\sum_{i\ne j}c_i\overline{c_j}\varphi(X_{E_i}X_{E_j})+\sum_{i=1}^n|c_i|^2\varphi(X_{E_i}^2)\\
=&\sum_{i\ne j}c_i\overline{c_j}\varphi(X_{E_i})\varphi (X_{E_j})+\sum_{i=1}^n|c_i|^2(\kappa_2(X_{E_i})+\varphi(X_{E_i})^2)\\
=&\int_\mathbb{R}\overline{s(x)}dx\int_\mathbb{R}s(x)dx+\int_{\mathbb{R}}|s(x)|^2dx\\
\le & \|s\|_1^2+\|s\|_2^2\le (|supp(s)|^2+1)\|s\|_2^2,
\end{align*}
where $|supp(s)|$ is the Lebesgue measure of the support set of function $s$. That is,
$$\varphi(X^*(s)X(s))\le \|s\|_1^2+\|s\|_2^2\le (|supp(s)|^2+1)\|s\|_2^2.\eqno (1)$$
Hence, for a function $f\in L^2(\mathbb{R})$ with finite measure
support $S_f$, we can define the integral of $f$ with respect to a
free Poisson random measure $X$ as follows. Let $\{s_n\}$ be a
sequence of simple functions such that $\|s_n-f\|_2\rightarrow 0$,
as $n\rightarrow \infty$, and the support of $s_n$, for all $n$, are
contained in $S_f$, the support of $f$. Then, we have
$\|X(s_n)-X(s_m)\|_2^2\le (|S_f|^2+1)\|s_m-s_n\|_2^2\rightarrow 0$,
as both $m$ and $n$ approach $\infty$. It implies that we can define
the integral as $$X(f):=\int_\mathbb{R}f(x)X(dx)=\lim_{n\rightarrow
\infty}X(s_n)\in L^2(\mathcal{A},\varphi),$$ where the limit is
taken in the norm $\|\cdot\|_2$ in $\mathcal{A}$. It is obvious that
the limit $X(f)$ is independent of the choice of $s_n$. Moreover,
for $f\in L^2(\mathbb{R})$ with $S_f:=supp(f), |S_f|<\infty$, let
$s_n\rightarrow f$  in $\|\cdot\|_2$, and $supp(s_n)\subseteq E_f$.
By Cauchy-Schwartz inequality, $s_n\rightarrow f$ with respect to
$\|\cdot\|_1$. It follows from $(1)$ that $$\varphi
(X(f)^*X(f))=\lim_{n\rightarrow \infty}\varphi
(X(s_n)^*X(s_n))\le\lim_{n\rightarrow
\infty}(\|s_n\|_1^2+\|s_n\|_2^2)=\|f\|_1^2+\|f\|_2^2. \eqno (2)$$

For a function $f\in L^2(\mathbb{R})$, we can define  $$X(f,t):=\int_{-t}^tf(x)X(dx), \forall t\ge 0, X(f)=\lim_{t\rightarrow \infty}X(f,t),\eqno (3)$$ provided that the limit taken in the norm $\|\cdot\|_2$ in $L^2(\mathcal{A},\varphi)$ exists. By the above definition of integration, we have the following rudimental property.
\begin{Proposition} Let $f,g\in L^2(\mathbb{R})$ such that $X(f)$ and $X(g)$ exist. Then, for real numbers $\alpha, \beta$,
$$X(\alpha f+\beta g)=\alpha X(f)+\beta X(g).$$
\end{Proposition}

\begin{Theorem} Let $f\in L^1(\mathbb{R})\cap L^2(\mathbb{R})$ be a real-valued function. Then $ X(f)=\lim_{t\rightarrow \infty}X(f,t)$ exists.
\end{Theorem}
\begin{proof} For $f\in L^1(\mathbb{R})\cap L^2(\mathbb{R})$, let $I_n=[-n,n], n=1, 2, \cdots$. Then $I_n\nearrow\mathbb{R}$, and  $$\int_{\mathbb{R}}|f(x)|dx=\lim_{n\rightarrow \infty}\int_{I_n}|f(x)|dx, \int_{\mathbb{R}}|f(x)|^2dx=\lim_{n\rightarrow \infty}\int_{I_n}|f(x)|^2dx.$$ Thus, let $f_n(x)=f(x)\chi_{[-n,n]}(x), n=1, 2, \cdots$. Then $\|f-f_n\|_1+\|f-f_n\|_2\rightarrow 0$, as $n\rightarrow \infty$. By identity $(2)$, we have
$$\|X(f_n)-X(f_m)\|_2^2=\|X(f_n-f_m)\|_2^2\le \|f_n-f_m\|_1^2+\|f_n-f_m\|_2^2\rightarrow 0,$$ as $n,m\rightarrow \infty$. Therefore, we have
$$X(f)=\int_{\mathbb{R}}f(x)X(dx)=\lim_{n\rightarrow \infty}\int_{-n}^nf(x)X(dx)$$ exists.
\end{proof}

\begin{Theorem}
Let $f_n, n\ge 1, f$ be real-valued functions in $L^1(\mathbb{R})\cap L^2(\mathbb{R})$ such that $\lim_{n\rightarrow \infty}(\|f_n-f\|_1+\|f_n-f\|_2)=0$.Then $\lim_{n\rightarrow \infty}\|X(f_n)-X(f)\|_2=0$, that is,
$$\lim_{n\rightarrow \infty}\int_\mathbb{R}f_n(t)dX_E(t)=\int_\mathbb{R}\lim_{n\rightarrow\infty} f_n(t)dX_E(t).$$
\end{Theorem}
\begin{proof}
For every $N\in \mathbb{N}$, we have
$$\lim_{n\rightarrow\infty}(\int_{[-N,N]}|f_n(t)-f(t)|+|f_n(t)-f(t)|^2)dt\le \lim_{n\rightarrow\infty}(\int_{\mathbb{R}}|f_n(t)-f(t)|+|f_n(t)-f(t)|^2)dt=0.$$
For every $\varepsilon>0$, there exists a number $n_0\in \mathbb{N}$ such that $$\|f_{n,N}-f_N\|_1+\|f_{n,N}-f_N\|_2<\varepsilon,\forall n\ge n_0, N\in \mathbb{N}.$$ Let $f_{n,N}=\chi_{[-N,N]}(t)f_n(t)$ and $f_{N}(t)=\chi_{[-N,N]}(t)f(t)$. By (2), we have
$$\|X(f_{n, N})-X(f_N)\|_{L^2(\A, \varphi)}\le \|f_{n,N}-f_N\|_1+\|f_{n,N}-f_N\|_2<\varepsilon,  \forall n\ge n_0, N\in \mathbb{N}.$$
Hence, for $\varepsilon>0$, we have  $$\|X(f_n)-X(f)\|_{L^2(\A,\varphi)}=\lim_{N\rightarrow \infty}\|X(f_{n, N})-X(f_N)\|_{L^2(\A, \varphi)}\le \varepsilon, \forall n\ge n_0.$$
It follows that $\lim_{n\rightarrow \infty}\|X(f_n)-X(f)\|_{L^2(\A, \varphi)}=0.$
\end{proof}
\begin{Remark}
\begin{enumerate}
\item Let $a$ be a self-adjoint operator in a $W^*$-probability space $(\mathcal{A},\varphi)$ with a free Poisson distribution $\kappa_n(a)=\lambda>0, n=1, 2, \cdots$. If $0<\lambda<1$, by Remark 4.3 and Section 4.2 in \cite{MA}, we can choose a positive element $sps$ having the same distribution, where $s$ is a self-adjoint operator with the standard semicircle distribution, and $p$ is a projection in $\mathcal{A}$ such that $s$ and $p$ are freely independent, and $\varphi(p)=\lambda$. Moreover, for a general $\lambda>0$, by the construction in section 4.2 in \cite{MA}, we still can choose a positive operator $a'\in \mathcal{A}$ such that $a'$ and $a$ have the same distribution.
\item We know that if $\kappa_m(a)>0, \forall m\ge 1$, then $\varphi (a^m)>0, \forall m\ge 1$ ($(11.6)$ and $(11.8)$ in \cite{NS}). But that having all positive moments doesn't mean that $a\in \A$ is positive. For instance, let $\A=L^\infty ([0,1],dx)$, $E_1=[0,1/3], E_2=[2/3,1]$, and $f=2\chi_{E_1}-\chi_{E_2}\in \A$ is not non-negative. But $\int_0^1f(x)^ndx=1/3(2^n+(-1)^n)>0,\forall n=1, 2, \cdots.$
\end{enumerate}
\end{Remark}

\begin{Theorem} If $\{X_E:E\in \mathcal{B}_0\}$ is a free Poisson random measure and $X_E\ge 0,\forall E\in \mathcal{B}_0$, then the integration with respect to the random measure is a contractive mapping from $L^1_\mathbb{R}(\mathbb{R})$ into $L^1(\mathcal{A},\varphi)$, where $L^1_\mathbb{R}(\mathbb{R})$ is the space of all real-valued $L^1$-functions on $\mathbb{R}$.
\end{Theorem}
\begin{proof} Let $s=\sum_{i=1}^nc_i\chi_{E_i}$ be a simple function with finite support $|\cup_{i=1}^nE_n|=\sum_{i=1}^n|E_i|<\infty$. By the definition, $X(s)=\sum_{i=1}^nc_iX_{E_i}$. Thus, $$\varphi(|X(s)|)\le \sum_{i=1}^n|c_i|\varphi (X_{E_i})=\sum_{i=1}^n|c_i||E_i|=\|s\|_1.\eqno (4)$$
Let $f\in L^1_\mathbb{R}(\mathbb{R})$, and choose a sequence $\{s_n:n\ge 1\}$ of simple functions such that $f=\lim_{n\rightarrow \infty}s_n$ a. e. and in $L^1$. Then, by (4), we define $$X(f)=\int_{\mathbb{R}}f(x)X(dx):=\lim_{n\rightarrow \infty}X(s_n),$$ where the limit is taken in $L^1(\mathcal{A},\varphi)$. Moreover, $$\|X(f)\|_1=\varphi(|X(f)|)=\lim_{n\rightarrow\infty}\varphi(|X(s_n)|)\le \lim_{n\rightarrow\infty}\|s_n\|_1=\|f\|_1.$$
\end{proof}
When we consider real-valued functions in $L^1(\mathbb{R})$, the integration operator $X:L^1(\mathbb{R})\rightarrow L^1(\mathcal{A},\varphi)$ has the following monotonic property.
\begin{Proposition} Suppose that $X_E\ge 0,\forall E\in \mathcal{B}_0$.  If real-valued functions $f\le g$, and $f, g\in L^1(\mathbb{R})$, then $X(f)\le X(g)$ in $L^1(\mathcal{A},\varphi)$.
\end{Proposition}
To study the distribution of $X(f)$, we focus on real-valued functions in $L^{\infty -}:=\bigcap_{n\ge 1} L^n(\mathbb{R})$.
\begin{Theorem} Let $f\in L^{\infty -}$ be a real-valued function. If $X_E\ge 0,\forall E\in \mathcal{B}_0$, then $X(f)\in \bigcap_{n=1}^\infty L^n(\mathcal{A},\varphi)$, and $X(f)$ has a compound free Poisson distribution: $$\kappa_m(X(f))=\int_\mathbb{R}f(x)^mdx, m=1, 2, \cdots.$$
\end{Theorem}
\begin{proof}
Suppose first that $f\ge 0$. Then there exists a sequence $\{s_n:n\ge 1\}$ of simple functions such that $s_n\rightarrow f$ a. e., as $n\rightarrow \infty$, and $0\le s_n(x)\le f(x)$ a. e.. By Lebesgue's dominated convergence theorem, $\lim_{n\rightarrow \infty}s_n=f$ in $L^m$, for all $m\ge 1$. It implies that $$\lim_{n\rightarrow \infty}\int_{\mathbb{R}}s_n(x)^mdx=\int_{\mathbb{R}}f(x)^mdx, m=1,2,\cdots.\eqno (5)$$

On the other hand, for a simple function $s=\sum_{i=1}^nc_i\chi_{E_i}$, $X(s)=\sum_{i=1}^nc_iX_{E_i}$. It follows that $$\kappa_m(X(s))=\kappa_m(\sum_{i=1}^nc_iX_{E_i})=\sum_{i=1}^mc_i^m\kappa_m(X_{E_i})=\sum_{i=1}^nc_i^m|E_i|=\int_{\mathbb{R}}s^m(x)dx.\eqno (6)$$
By (5) and (6), we have
$$\lim_{n\rightarrow \infty}\kappa_m(X(s_n))=\int_{\mathbb{R}}f(x)^mdx, m=1,2,\cdots.\eqno (7)$$
Moreover, for a real-valued simple function
$s=\sum_{i=1}^nc_i\chi_{E_i}$, $|s(x)|=\sum_{i=1}^n|c_i|\chi_{E_i}$.
Since $\lim_{n\rightarrow \infty}s_n=f$ in $L^m$, for all $m\ge 1$,
we have $\int_{\mathbb{R}}|s_{n_1}-s_{n_2}|^m(x)dx\rightarrow
\infty$, as $n_1, n_2\rightarrow \infty$, for all $m\ge 1$. Let
$s_{n_1}-s_{n_2}=\sum_{i=1}^kd_i\chi_{F_i}$, we have
$$\int_{\mathbb{R}}|s_{n_1}-s_{n_2}|^m(x)dx=\sum_{i=1}^k|c_i|^m|E_i|\rightarrow
0,$$ as $n_1, n_2\rightarrow \infty$, for all $m\ge 1$. Therefore,
by (6) we have
$$\kappa_m(|X(s_{n_1}-s_{n_2})|)=\sum_{i=1}^k|c_i|^m|E_i|=\|s_{n_1}-s_{n_2}\|_m^m\rightarrow 0, $$ as, $n_1, n_2\rightarrow \infty$, for all $m\ge 1$.
By the moment-cumulant formals (11.7) and (11.8) in \cite{NS}, we have
 $$\varphi(|X(s_{n_1})-X(s_{n_2})|^m)\rightarrow 0,$$ as $n_1, n_2\rightarrow \infty$, for all $m\ge 1$.
 Note that $0\le X(s_n)\in\A$. Thus, there is a positive operator $X(f)\in L^{\infty -}:=\bigcap_{n=1}^\infty L^n(\A,\varphi)$ such that
$\varphi(X(f)^m)=\lim_{n\rightarrow \infty}\varphi(X(s_n)^m)$. It implies from (7) that $$\kappa_m(X(f))=\lim_{n\rightarrow\infty}\kappa_m(X(s_n))=\int_{\mathbb{R}}f(x)^mdx, m=1,2,\cdots.$$
By Exercise 16.21 in \cite{NS}, there is a non-commutative probability space $(\mathcal{D}, \psi)$ and an element $d\in \mathcal{D}$ such that $\psi(d^m)=\int_{\mathbb{R}}f(x)^mdx, m=1,2,\cdots.$ Then we have $$\kappa_m(X(f))=\psi(d^m), m=1, 2, \cdots,$$
that is, $X(f)$ has a compound Poisson distribution with $\lambda=1$ and measure $\nu$, where $m_n(\nu)=\psi (d^n), n=1, 2, \cdots$.

For a general real-valued $f\in L^{\infty -}$, let $f=f^+-f^-$. Then $|f|=f^++f^-$. We have $X(f)=X(f^+)-X(f^-)\in L^{\infty -}(\A,\varphi)$. Let simple functions $s_n^+\rightarrow f^+$ and $s_n^-\rightarrow f^-$ a. e.. Then $X(s_n^+)$ and $X(s_n^-)$ are freely independent, since $supp(s_n^+)\subset supp(f^+)$ and $supp(s_n^-)\subset supp(f^-)$ are disjoint. It follows that
\begin{align*}
\kappa_m(X(f))&=\kappa_m(X(f^+)-X(f^-))=\lim_{n\rightarrow \infty}\kappa_m(X(s_n^+)-X(s_n^-))\\
=&\lim_{n\rightarrow\infty}\kappa_m(X(s_n^+))+(-1)^m\lim_{n\rightarrow\infty}\kappa_m(X(s_n^-))\\
=&\int_{\mathbb{R}}f^+(x)^mdx+(-1)^m\int_{\mathbb{R}}f^-(x)^mdx=\int_{\mathbb{R}}f(x)^mdx.
\end{align*}
It implies that $X(f)$ has a compound free Poisson distribution $$\kappa_m(X(f))=\int_\mathbb{R}f(x)^mdx, m=1, 2, \cdots.$$
\end{proof}
\section{Centered Free Poisson Measures}

A random variable $a$ in a non-commutative probability space $(\A,\varphi)$ is {\sl centered} if $\varphi (a)=0$. For a random variable $a\in \A$, $\widetilde{a}:=a-\varphi (a)$ is always centered.

\begin{Definition}
\begin{enumerate}
\item If a random variable $a\in (\A,\varphi)$ has a  free Poisson distribution $\kappa_n(a)=\lambda\alpha^n, n\ge 1$, then we say $\widetilde{a}=a-\varphi(a)$ has a centered free Poisson distribution, that is, $\kappa_n(\widetilde{a})=\lambda\alpha^n, n\ge 2$, and $\kappa_1(\widetilde{a})=0$.
\item If $X_E, E\in \mathcal{B}_0$, is a free Poisson random measure, we called $\widetilde{X}_E=X_E-|E|, E\in \mathcal{B}_0$, a centered free Poisson random measure.
\end{enumerate}
\end{Definition}

\begin{Proposition} A random measure  $\widetilde{X}_E, E\in \mathcal{B}_0$, is a centered free Poisson random measure, if and only if
\begin{enumerate}
\item $\widetilde{X}_E$ has a centered free Poisson distribution with parameter $|E|$, for a set $E\in \mathcal{B}_0$.
\item If $E_1, E_2, \cdots, E_k$ are mutually disjoint, then $\widetilde{X}_{E_1}, \widetilde{X}_{E_2}, \cdots, \widetilde{X}_{E_k}$ are freely independent.
\end{enumerate}
\end{Proposition}
\begin{proof}
Suppose that $\widetilde{X}_E=X_E-|E|$. Then by Proposition 11.15 in \cite{NS}, for $a_1, a_2, \cdots, a_m\in \A, m\ge 2$,  we have $\kappa_m(a_1, a_2, \cdots, a_m)=0$, if there is at least one scalar element $a_i=\alpha I$, where $\alpha$ is a constant, and $I$ is the unit in $A$. It follows that, for $m\ge 2, m\in \mathbb{N}$,
$$\kappa_m(\widetilde{X}_E)=\kappa_m(X_E-|E|,\widetilde{X}_E, \cdots, \widetilde{X}_E)=\kappa_m(X_E,\widetilde{X}_E,\cdots,\widetilde{X}_E)=\cdots=\kappa_m(X_E)=|E|.$$
It is obvious that $\{\widetilde{X}_{E_1}, \cdots, \widetilde{X}_{E_n}\}$ is a free family if $E_1, E_2, \cdots, E_n$ are mutually disjoint.

Conversely, if $\widetilde{X}_E$ satisfies the two conditions, let $X_E=\widetilde{X}_E+|E|$. By the above discussion, $X_E$ is a free Poisson random measure.
\end{proof}
Let $s=\sum_{i=1}^nc_i\chi_{E_i}$ with $|E_i|<\infty$. Define $\widetilde{X}(s)=\int_{\mathbb{R}}s(x)\widetilde{X}(dx):=\sum_{i=1}^nc_i\widetilde{X}_{E_i}$.
\begin{Lemma}$\|\widetilde{X}(s)\|_2^2=\|s\|_2^2.$
\end{Lemma}
\begin{proof}
\begin{align*}
\|\widetilde{X}(s)\|_2^2&=\varphi (\widetilde{X}(s)^*\widetilde{X}(s))=\varphi(\sum_{i,j=1}^nc_i\overline{c_j}\varphi (\widetilde{X}_{E_i}\widetilde{X}_{E_j}))\\
=&\sum_{i=1}^n|c_i|^2\varphi(\widetilde{X}_{E_i}^2)=\sum_{i=1}^n|c_i|^2(\kappa_2(\widetilde{X}_{E_i})+\varphi (\widetilde{X}_{E_i})^2)\\
=&\sum_{i=1}^n|c_i|^2|E_i|=\|s\|_2^2.
\end{align*}
\end{proof}
From the above lemma, we can extend the integration to $L^2(\mathbb{R})$. Let $f$ be a real-valued functionin  $ L^2(\mathbb{R})$, and $\{s_n:n\ge 1\}$ be a sequence of simple functions such that $s_n\rightarrow f$ a. e. and in $\|\cdot\|_2$. Then define $\widetilde{X}(f)=\lim_{n\rightarrow \infty}\widetilde{X}(s_n)\in L^2(\A,\varphi)$, where the limit is taken with respect to  $\|\cdot\|_2$ of $L^2(\A,\varphi)$. The operator $\widetilde{X}:L^2(\mathbb{R})\rightarrow L^2(\A,\varphi)$ is isometric.
\begin{Lemma}
Let $s=\sum_{i=1}^nc_iE_i$ be a real-valued simple function, and $X_E\ge 0, \forall E\in \B_0$. Then $\|\widetilde{X}(s)\|_1\le 2\|s\|_1$.
\end{Lemma}
\begin{proof}
$$\|\widetilde{X}(s)\|_1=\varphi(|\widetilde{X}(s)|)\le \sum_{i=1}^n|c_i|\varphi (X_{E_i})+\sum_{i=1}^n|c_i||E_i|=2\|s\|_1.$$
\end{proof}
Hence, we can extend the integration to $L^1(\mathbb{R})$. For a real-valued $f\in L^1(\mathbb{R})$, let $s_n$ be a real-valued simple functions such that $s_n\rightarrow f$ a. e. and in $L^1(\mathbb{R})$. Then we define $\widetilde{X}(f)=\lim_{n\rightarrow \infty}\widetilde{X}(s_n)\in L^1(\A,\varphi)$, where the limit is taken with respect to $\|\cdot\|_1$ of $L^1(\A,\varphi)$.


\begin{thebibliography}{99}
\bibitem[AA]{AA}A. Alexanderian. {\sl A brief note on the Karhunen-Loeve expansion}.\\ http://users.ices.utexas.edu/~alen/articles/KL.pdf.
\bibitem[MA]{MA}M. Anshelevich. {\sl Free Stochastic measures via noncrossing partitions}. Adv. Math. 155(2000), No.1, 154-179.
\bibitem[MA1]{MA1}M. Anshelevich. {\sl Ito formula for free stochastic integrals}. J. Funct. Anal., 188(2002), 292-315.
\bibitem[MA2]{MA2}M. Abshelevich. {\sl Linearization coefficients for orthogonal polynomials using stochastic processes}. Ann. of Probab., 33(2005), No.1, 114-136.
\bibitem[PB]{PB}P. Bartlett. {\sl Reproducing Kernel Hilbert spaces}.\\
http://www.cs.berkeley.edu/~bartlett/courses/281b-sp08/7.pdf.
\bibitem[BnT]{BnT}O. E. Barndorff-Nielson and S. Thorjornsen. {\sl Classical and free infinite divisibility and Levy processes}. Lecture Note in Math. 1866 (2006), 33-160. Springer-Verlarg Berlin-Heidelberg.
\bibitem[PBi]{PBi}P. Biane. {\sl Free Brownian motion, free stochastic calculus and random matrices}. Free Probability Theory edited by D. Voiculescu, 1-19, Fields Institute Communications 12, AMS, 1997.
\bibitem[BS1]{BS1}P. Biane and R. Speicher. {\sl Stochastic calculus with respect to free Brownian motion and analysis on Wigner spaces}. Probab. Theory Related Fields 112(1998), no.3, 373-409.
\bibitem[BS2]{BS2} P. Biane and R. Speicher.{\sl Diffusion, free entropy, and free Fisher information}. Ann. Inst. H. Poincare Probab. Stat. 25 (2001), 581-606.
\bibitem[BP]{BP}S. Bourguin and G. Peccati. {\sl Semicircle limits on the free Poisson chaos: counterexample to a transfor principle}. J. Funct. Anal., 267(2014), No. 4, 963-997.
\bibitem[RG]{RG}R.G. Gallager. {\sl Stachastic Processes: Theory for Applications}. Cambridge University Press, 2013.
\bibitem[MG1]{MG1}M. Gao. {\sl Free Markov processes and Stochastic differential equations in von Neumann algebras}. Illinois J. Math. 52(2008), No. 1, 153-180.
\bibitem[MG2]{MG2}M. Gao. {\sl Free Ornstein-Uhlenbeck processes}. J. Math. Anal. Appl., 322(2006), 177-192.
\bibitem[GSS]{GSS}P. Glockner, M. Schurmann and R. Speicher. {\sl Realization of free white noises}. Arch. Math., Vol58(1992), 407-416.
\bibitem[DJ]{DJ}D. Johnson. {\sl Statistical Signal Processing}.\\ http://elec531.blogs.rice.edu/files/2015/01/notes1.pdf.
\bibitem[KR]{KR}R. Kadison and J. Ringrose. {\sl Fundamentals of the theory of operator algebras}. Graduate Studies in MAth. Vol. 16, AMS, 1997.
\bibitem[TK]{TK} T. Kurts. {\sl Lecures on Stochastic Analysis}.\\ http://www.math.wisc.edu/~kurtz/735/main735.pdf.
\bibitem [NS]{NS}A. Nica and R. Speicher. {\sl Lectures on Combinatorics for Free Probability}, LMS Lecture Notes 335, Cambridge University Press, 2006.
\bibitem[NS1]{NS1} A. Nica and R. Speicher. {\sl On the multiplication of free $N$-tuples of noncommutative random variables}. Amer. J. of Math., 118(1996), 799-837.
\bibitem[VDN]{VDN}D. Voiculescu, K. Dykema, and A. Nica. {\sl Free Random variables}. CRM Monograph Series, Vol. 1, AMS, 1992.
\end{thebibliography}
\end{document}